\documentclass[11pt]{article}
\usepackage{amsfonts}
\usepackage{amscd}
\usepackage{amsmath,amstext,amsthm,amsbsy,amssymb}
\usepackage{mathrsfs}
\usepackage{multirow}
\usepackage{epsfig}
\newcommand{\bi}{{\bf i}}
\newcommand{\bj}{{\bf j}}
\newcommand{\bk}{{\bf k}}
\newcommand{\bc}{{\mathbb C}}
\newcommand{\br}{{\mathbb R}}
\newcommand{\bh}{{\mathbb H}}

\newtheorem{thm}{Theorem}[section]
\newtheorem{lem}{Lemma}[section]
\newtheorem{pro}{Proposition}[section]

\newtheorem{rem}{Remark}[section]
\newtheorem{exam}{Example}[section]
\newtheorem{defi}{Definition}[section]

\begin{document}

\title{Quadratic unilateral polynomials over split quaternions}
\author{ Wensheng Cao \\
School of Mathematics and Computational Science,\\
Wuyi University, Jiangmen, Guangdong 529020, P.R. China\\
e-mail: {\tt wenscao@aliyun.com}}
\date{}
\maketitle

\bigskip
{\bf Abstract} \,\,   In this paper,  we derive
explicit formulas for computing the roots of $ax^{2}+bx+c=0$ with $a$ being not invertible in split  quaternion algebra. We also use  the approach developed by Opfer, Janovska and Falcao etc. to verify our results when the corresponding companion polynomials $C(x)\neq 0$.

\vspace{2mm}\baselineskip 12pt



\section{Introduction}
\subsection{Split quaternions}
Let $\br, \bc, \bh$ or $\bh_s$ be respectively the set of  real numbers, complex numbers, quaternions or split quaternions.   The quaternion algebra $\bh$  and split quaternion algebra $\bh_s$  are non-commutative extensions of the complex numbers.
$\bh_s$ can be represented as
$$\bh_s=\{x=x_0+x_1\bi+x_2\bj+x_3\bk,x_i\in \br,i=0,1,2,3\},$$
where $1,\bi,\bj,\bk$ are basis of $\bh_s$ satisfying the following multiplication rules:
\begin{table}[h]  \centering
	\caption{The multiplication table of $\bh_s$}
	\vspace{3mm}
	\begin{tabular}{c|cccc}
	$\bh_s$	&1& $\bi$ &  $\bj$  &  $\bk$\\
		\hline
		1 & 1& $\bi$ &  $\bj$  &  $\bk$\\
		$\bi$ &$\bi$ &-1 & $\bk$ & -$\bj$ \\
		$\bj$&$\bj$&-$\bk$ &1& -$\bi$  \\
		$\bk$	&$\bk$&$\bj$ & $\bi$ & 1
	\end{tabular}
\end{table}

Let $\bar{x}=x_0-x_1\bi-x_2\bj-x_3\bk$, $$\Re(x)=(x+\bar{x})/2=x_0,\ \Im(x)=(x-\bar{x})/2=x_1\bi+x_2\bj+x_3\bk$$ be  respectively the conjugate, real part and  imaginary part of $x\in\bh_s$.
Each  element in $\bh_s$ can be expressed  as
$$x=(x_0+x_1\bi)+(x_2+x_3\bi)\bj =z_1+z_2\bj,\mbox{ where }z_1=x_0+x_1\bi,z_2=x_2+x_3\bi\in \bc.$$
This implies that $$\bh_s=\bc\oplus\bc\bj \mbox{ and }\bj z=\bar{z}\bj, \forall z\in \bc.$$
For $x=x_0+x_1\bi+x_2\bj+x_3\bk=z_1+z_2\bj\in \bh_s$, we define
\begin{equation}\label{Ix}I_x=\bar{x}x=x\bar{x}=x_0^2+x_1^2-x_2^2-x_3^2=|z_1|^2-|z_2|^2\end{equation}
and  \begin{equation}\label{kp}K(x)=\Im(x)^2=-x_1^2+x_2^2 + x_3^2=\Re(x)^2-I_x.\end{equation}
It can be easily verified that \begin{equation} \overline{xy}=\bar{y}\bar{x},\  I_{yx}=I_yI_x, \forall x,y\in\bh_s.\end{equation}
 Unlike the Hamilton quaternion algebra, the split quaternion algebra contains nontrivial zero divisors.
The set of zero divisors is denoted by  \begin{equation}
	Z(\bh_s)=\{x\in \bh_s:I_x=0\}.
\end{equation}
Note that   $$\Re(xy)=\Re(yx)=x_0y_0-x_1y_1+x_2y_2+x_3y_3,\forall x,y\in \bh_s.$$
For $x=x_0+x_1\bi+x_2\bj+x_3\bk,y=y_0+y_1\bi+y_2\bj+y_3\bk\in \bh_s$,
we  define    \begin{equation}\label{inp2}\left\langle x,y \right\rangle=x_0y_0+x_1y_1-x_2y_2-x_3y_3.\end{equation}
 For the sake of simplification, we  denote $$\left\langle x,y \right\rangle:=P_{x,y}=P_{xy}.$$
Accordingly we have
$$I_x=\left\langle x,x\right\rangle=P_{xx},\ \Re({\bar{y}x})=\Re({\bar{x}y})
=\Re({y\bar{x}})=\left\langle x,y \right\rangle=P_{xy}=P_{yx}.$$
	We say that two split quaternions $a,b\in \bh_s$  are similar
if and only if there exists a $q\in \bh_s-Z(\bh_s)$ such that $qa=bq$. 
Similarity is an  equivalence relation in split quaternions.  Let 
\begin{equation}[\lambda]_{\sim}=\{q^{-1}\lambda q:\forall q\in \bh_s\}\end{equation}
be the similarity class of $\lambda$.
\begin{pro}(c.f. \cite[Theorem 4.3]{cao})\label{thmsim}
	Two split quaternions $a,b \in \bh_s-\br$ are similar if and only if
	$$\Re(a)=\Re(b),\, K(a)=K(b).$$
\end{pro}
We mention that two split quaternions $a_0$ and $a_0+\bi+\bj$ are not similar.
We define the quasisimilar class of $q$ as the following set \begin{equation}[q]=\{p\in \bh_s: \Re(p)=\Re(q), I_p=I_q\}.\end{equation}

\begin{pro}
	If $q\in \br$ then $[q]_{\sim}=\{q\}\subset [q]$; if $q\in \bh_s-\br$ then $[q]_{\sim}=[q].$ 
\end{pro}

\subsection{Niven's algorithm over  $\bh_s$}

The problem of solving quaternionic quadratic equations for quaternions  was  first approached by Niven \cite{niven}. In \cite{Opfer10,sero01}, Niven’s algorithm is tailored for finding zeros of unilateral
polynomials $$f(x)=\sum_{j=0}^{n}a_jx^j, z,a_j\in \bh, \mbox{ where } a_n=1, a_0\neq 0.$$

Recently, Niven's  algorithm has been developed by  Falcao, Opfer and Janovska etc.\cite{Ireneu,Irene,Opfer17,Opfer18} to solve the unilateral polynomials over $\bh_s$. We sketch Niven's algorithm over  $\bh_s$ as follows.

In the set of   polynomials of the form
\begin{equation}f(x)=\sum_{j=0}^{n}a_jx^j, a_j\in \bh_s,\end{equation}
we define the addition and multiplication of such polynomials as in the commutative case where the variable commutes with the coefficients.
With  the two  operations, this set becomes a ring, referred to as the ring of unilateral polynomials in $\bh_s$ and denoted by $\bh_s[x]$. For $f(x)\in \bh_s[x]$, we define the $$\overline{f(x)}=\sum_{j=0}^{n}\overline{a_j}x^j.$$
The companion  polynomial $C(x)$ of $f(x)$ is defined as
\begin{equation}C(x)=\sum_{j,k=0}^{n}\overline{a_j}a_k\,x^{j+k}.\end{equation}
Note that in the two operations of the ring $\bh_s[x]$, \begin{equation}C(x)=f(x)\overline{f(x)}=\overline{f(x)}f(x).\end{equation}
We mention that $C(x)$ is a polynomial with real coefficients of degree at most $2n$.

For each quasisimilar  class $[q]$, we define real coefficient quadratic polynomial
\begin{equation}\Psi_{[q]}(x)=x^2-2\Re(q)x+I_q.\end{equation}
Obviously \begin{equation}\Psi_{[q]}(p)=0,\forall p\in [q].\end{equation}

In \cite{Ireneu,Irene,Opfer17,Opfer18}, Falcao, Opfer and Janovska etc. considered the unilateral
polynomials
\begin{equation}\label{pz}f(x)=\sum_{j=0}^{n}a_j x^j, x,a_j\in \bh_s, \mbox{ where } a_n, a_0 \mbox{  are invertible}.
\end{equation}

The mechanism of Niven's algorithm can be described by the following proposition.

\begin{pro}\label{pro1.1}(c.f.\cite[Theorem 3.7]{Irene}) Let \begin{equation}Z(f)=\{q\in\bh_s:f(q)=0\}.\end{equation}
	If $q\in Z(f)$,  then $\Psi_{[q]}(x)$ is a divisor of $C(x)$ in complex number field. That is $$C(x)=h(x)\Psi_{[q]}(x),\,\, h(x)\in \bc[x].$$
\end{pro}

 For such a $[q]$, one can perform the following
$$f(x)=h(x)\Psi_{[q]}(x)+A_q+B_qx.$$
One then can say that the element $y\in [q]$ satisfying $A_q+B_qy=0$ belongs to $Z(f)$.
That is $$Z(f)=\bigcup_{[q]}\{y\in [q]:A_q+B_qy=0\}.$$

 The essential principle of Niven's algorithm using the companion polynomial is that   we can figure out the $\Re(q)$ and $I_q$ by the companion polynomial  and then search for the solutions in these quasisimilar classes $[q]$.

\subsection{Quadratic equation in $\bh_s$}

 The   quadratic equation $$ax^2+bx+c=0,a,b,c\in\bh_s$$ 
 can be reformulated as $$x^2+a^{-1}bx+a^{-1}c=0$$ if $a$ is invertible. Such quadratic equation has been considered in \cite{caoaxiom,Ireneu,Irene,Opfer18,Opfer17}.
    
  In this paper, we will focus on deriving explicit formulas of the roots of  the quadratic equation
\begin{equation}\label{eqcons}ax^2+bx+c=0, a\in Z(\bh_s)-\{0\}.\end{equation}
The  companion  polynomial of $f(x)$ is
$$C(x)=(a\bar{b}+b\bar{a})x^3+(a\bar{c}+c\bar{a}+I_b)x^2+(c\bar{b}+b\bar{c})x+I_c=0.$$
That is
\begin{equation}\label{czpoly}C(x)=2P_{ab}x^3+(2P_{ac}+I_b)x^2+2P_{bc}x+I_c=0.\end{equation}

\begin{exam} The  companion  polynomial of $$f(x)=(1+\bj)x^2 +(\bi-\bk)x-1+\bi-\bj-\bk= 0$$ is $C(x)\equiv 0$.
\end{exam}

The above example shows that we need to face the intricate  problem (\ref{eqcons}) without the help of the corresponding companion  polynomial.
Since $a$ is not invertible and $c$ is arbitary, the above quadratic equation has not been considered in  \cite{caoaxiom,Ireneu,Irene,Opfer18,Opfer17}.

To reduce the number of  parameters in $ax^2+bx+c=0$  and  simplify our consideration, we  have the following proposition.
\begin{pro}\label{reducepro2}(c.f.\cite[Section 7]{Opfer14}) The quadratic equation $dy^2+ey+f=0$ with  $d=d_1+d_2\bj\in Z(\bh_s)-\{0\}$,
	$d_1,d_2\in \bc$ is solvable if and only if the quadratic equation $$ax^2+bx+c=0$$ is solvable, where
	$$d_1^{-1}e=k_0+k_1\bi+k_2\bj+k_3\bk, k_i\in \br,i=0,\cdots,3$$ and
	$$a=1+d_1^{-1}d_2\bj,\, b=d_1^{-1}e-k_0(1+d_1^{-1}d_2\bj),\,
	c=d_1^{-1}f-\frac{d_1^{-1}ek_0}{2}+\frac{(1+d_1^{-1}d_2\bj)k_0^2}{4}.$$
	If the quadratic equation $ax^2+bx+c=0$ is solvable and $x$ is a  solution then
	$y= x-\frac{k_0}{2}$ is a solution of $dy^2+ey+f=0$.
\end{pro}

\begin{proof}
	Since $d=d_1+d_2\bj\in Z(\bh_s)-\{0\}$, we have $I_{d_1}=I_{d_2}\neq 0$ and $d_1$ is invertible. Hence
	$dy^2+ey+f=0$ is equivalent to $$(1+d_1^{-1}d_2\bj)y^2+d_1^{-1}ey+d_1^{-1}f=0.$$ Let $y=x-\frac{k_0}{2}.$
	Then $dy^2+ey+f=0$ is equivalent to
	$$(1+d_1^{-1}d_2\bj)\big(x^2-k_0x+\frac{k_0^2}{4}\big)+d_1^{-1}e\big(x-\frac{k_0}{2}\big)+d_1^{-1}f=0.$$
	That is
	$$(1+d_1^{-1}d_2\bj)x^2+[d_1^{-1}e-k_0(1+d_1^{-1}d_2\bj)]x+d_1^{-1}f
	-\frac{d_1^{-1}ek_0}{2}+\frac{(1+d_1^{-1}d_2\bj)k_0^2}{4}=0.$$
	\end{proof}
In the process of the above proof, let $$d_1^{-1}d_2=a_2+a_3\bi, a_2,a_3\in \br.$$ Then we have $a_2^2+a_3^2=1$ and
	$$a=1+d_1^{-1}d_2\bj=1+a_2\bj+a_3\bk\in Z(\bh_s).$$
	Since $\Re[d_1^{-1}e-k_0(1+d_1^{-1}d_2\bj)]=0$,  we have $$b=d_1^{-1}e-k_0(1+d_1^{-1}d_2\bj)=b_1\bi+b_2\bj+b_3\bk.$$

Hence  we only need to solve the following equations:
\begin{itemize}
	\item  Equation I: \quad  $ax^2+c=0, a=1+a_2\bj+a_3\bk\in Z(\bh_s)$;
	\item  Equation II: \quad $ax^2+bx+c=0, a=1+a_2\bj+a_3\bk\in Z(\bh_s), b=b_1\bi+b_2\bj+b_3\bk\neq 0$.
\end{itemize}

We will show in Proposition \ref{proc=0} that if Equation I is solvable then the corresponding companion polynomial $$C(x)=0.$$
 By the   Moore-Penrose inverse property obtained in  \cite{cao}, we will solve Equations I  in Section 2 (Theorem \ref{thm2.1}).

For Equation II, observe  that
\begin{equation}x^2=x(2x_0-\bar{x})=2x_0x-I_x.
\end{equation}
Therefore  $ax^2+bx+c=0$  becomes \begin{equation}\label{linearx1} (2x_0a+b)x=aI_x-c.
\end{equation}
Let
\begin{eqnarray}\label{Nf1}N&=&I_x=\bar{x}x,\\
	\label{Tf1}T&=&\bar{x}+x=2x_0.\end{eqnarray}

By $(2x_0a+b)x=aI_x-c$, we have $(2x_0a+b)x\overline{(2x_0a+b)x}=(aI_x-c)\overline{aI_x-c}$.
That is
\begin{equation}\label{enf1} N(2TP_{ab}+I_b+2P_{ac})-I_c=0.\end{equation}

Any solutions $x=x_0+x_1\bi+x_2\bj+x_3\bk$ of $ax^2+bx+c=0$ must fall into two categories:
\begin{itemize}
	\item   $2x_0a+b\in Z(\bh_s)$;
	\item  $2x_0a+b\in \bh_s- Z(\bh_s)$.
\end{itemize}

For Equation  II, we define 
$$SZ=\{x\in \bh_s:ax^2+bx+c=0 \mbox{ and }2x_0a+b\in Z(\bh_s)\}$$ and
$$SI=\{x\in \bh_s:ax^2+bx+c=0 \mbox{ and }2x_0a+b\in \bh_s-Z(\bh_s)\}.$$

In order to solve Equation II,  for technical reasons,  we divide Equation II into the following two equations:
\begin{itemize}
	\item   Equation II for SZ,
	\item   Equation II for SI.
\end{itemize}

If $2x_0a+b\in \bh_s-Z(\bh_s)$, then by (\ref{linearx1}) we have
\begin{equation}x=(2x_0a+b)^{-1}(aI_x-c)=(Ta+b)^{-1}(aN-c)=\frac{(T\bar{a}+\bar{b})(aN-c)}{2TP_{ab}+I_b}\end{equation}
        and \begin{equation}\bar{x}=\frac{(N\bar{a}-\bar{c})(Ta+b)}{2TP_{ab}+I_b}.\end{equation}
Substituting  the above formulas of $x$ and $\bar{x}$ in  (\ref{Tf1}), we obtain
      \begin{eqnarray}
          \label{enf2}
     x+\bar{x}&=&\frac{-2TP_{ac}+2NP_{ab}-2P_{bc}}{2TP_{ab}+I_b}=T.
      \end{eqnarray}

Hence $(T,N)$ satisfies our first  real nonlinear  system:
      \begin{equation}\label{rsym1}\left\{
\begin{aligned}
   N(2TP_{ab}+I_b+2P_{ac})-I_c=0,\\
    2P_{ab}T^2+(2P_{ac}+I_b)T-2NP_{ab}+2P_{bc}=0.
  \end{aligned}
\right.\end{equation}

Since we aim to find a root of $ax^2+bx+c=0$, we do not know $x_0$ beforehand. For technical reason, we may  assume that  $$2x_0a+b=Ta+b\in \bh_s-Z(\bh_s).$$
For Equation II,  after obtaining  the pair $(T,N)$ from the real  nonlinear system (\ref{rsym1}), we need test  $Ta+b\in \bh_s-Z(\bh_s)$. Only for the pair  $(T,N)$ satisfying $Ta+b\in \bh_s-Z(\bh_s)$,
  we obtain the corresponding  solution $ x=(Ta+b)^{-1}(aN-c).$

If  $2x_0a+b\in Z(\bh_s)$ then
  \begin{equation}\label{ncasezero}\left\langle 2x_0a+b,2x_0a+b \right\rangle=4x_0P_{ab}+I_b=0.\end{equation}
Also we have
 \begin{equation}\label{nzacc}\left\langle aI_x-c,aI_x-c \right\rangle=-2I_xP_{ac}+I_c=0.\end{equation}
By Eq.(\ref{ncasezero}) and (\ref{nzacc}), we may know some information on $x_0$ and $I_x$.
 For example, if  $P_{ab}\neq 0$ then  $x_0=\frac{-I_b}{4P_{ab}}$; if $P_{ac}\neq 0$ then  $I_x=\frac{I_c}{2P_{ac}}$.
 However, in general,  we may get no  information  of $x_0$ and $I_x$. For example, if $P_{ab}=0,P_{ac}=0$.  We will resort  to its natural real nonlinear system as follows.

Let $a=1+a_2\bj+a_3\bk,$  $b=b_1\bi+b_2\bj+b_3\bk$,  $c=c_0+c_1\bi+c_2\bj+c_3\bk \in \bh_s$.
By the rule of multiplication Table 1, the  equation $ax^2+bx+c=0$ can be reformulated as our second real nonlinear
system:
\begin{equation}\label{rsym2}\left\{
\begin{aligned}
 & x_0^2-x_1^2+x_2^2+x_3^2+2a_2x_0x_2+2a_3x_0x_3-b_1x_1+b_2x_2+b_3x_3+c_0=0,\\
  &              2x_0x_1-2a_2x_0x_3+2a_3x_0x_2+b_1x_0-b_2x_3+b_3x_2+c_1=0,\\
 &2x_0x_2+a_2(x_0^2-x_1^2+x_2^2+x_3^2)+2a_3x_0x_1-b_1x_3+b_2x_0+b_3x_1+c_2=0,\\
 &2x_0x_3-2a_2x_0x_1+a_3(x_0^2-x_1^2+x_2^2+x_3^2)+b_1x_2-b_2x_1+b_3x_0+c_3=0.
\end{aligned}
\right.\end{equation}

Roughly speaking, we will rely on the two real systems (\ref{rsym1}) or (\ref{rsym2}) to solve the Equation II.  Under some conditions, we can deduce some  some linear relations of $x_i, i=0,\cdots,3$, which is helpfull in treating the Equation II.  For example, under the condition $P_{ab}=0$ and $x\in SZ$, we can deduce  some linear relations of $x_1,x_2$ and $x_3$ from Eqs.(\ref{rsym2}) (Proposition \ref{prop5.1}).

We list our  problem solving process in Table 2.

\begin{table}[!htbp]
	\centering
	\caption{ Problem solving process}
	\begin{tabular}{|c|c|c|c|}
		\hline
		Section &Type of Equation & Result&Examples of Theorem\\
		\hline
		 2&Equation I& Theorem 2.1&Example 2.1\\
\hline
 		\multirow{3}*{3}&Equation II with $P_{ab}\neq 0$ for SZ& Theorem 3.1&Examples 3.1 and 3.2\\
	\cline{2-4}
&Equation II with $P_{ab}=0,P_{\bi a,b}=0$ for SZ& Theorem  3.2&Examples 3.3 and 3.4\\
	\cline{2-4}
&Equation II with $P_{ab}=0,P_{a \bi,b}=0$ for SZ& Theorem  3.3&Examples 3.5 and 3.6\\
\hline
	\multirow{3}*{4}&Equation II with $P_{ab}\neq 0$ for SI& Theorem 4.1&Example 4.1\\
	\cline{2-4}
&Equation II with $P_{ab}=0,I_b+2P_{ac}\neq 0$ for SI& Theorem 4.2&Example  4.2\\
	\cline{2-4}
&Equation II with $P_{ab}=0,I_b+2P_{ac}=0$ for SI& Theorem 4.3&Example 4.3\\
	\cline{2-4}
\hline
	\end{tabular}
\end{table}

We remark that all examples in Table 2 are carefully chosen  to illustrate that all our formulas work.  The author has checked that the equation in Example 3.1 has no solution in $SI$. The equation of Example 3.2 (which is the same as  Example 4.1)   has a solution in $SZ$ and $SI$, respectively. We have given all solutions of all examples in our paper.

In Section 5, we will imitate the approach developed in \cite{Ireneu,Irene,Opfer18,Opfer17} using companion polynomials.  We apply such an approach to our examples in Table 2 with $C(x)\not\equiv 0$ and get the same results.  These examples provide a reciprocal authentication of the two approaches.

\section{Equation I}
 In this section, we consider $ax^2+c=0$ with $a=1+a_2\bj+a_3\bk\in Z(\bh_s)$.
    The  Moore-Penrose inverse \cite{cao} of  $a=t_1+t_2\bj,t_1,t_2\in \bc$ is defined to be
$$a^+=\left\{\begin{array}{ll}
0, & \hbox{if }\,\, $a=0;$ \\
\frac{\overline{t_1}-t_2\bj}{|t_1|^2-|t_2|^2}=\frac{\overline{a}}{I_a}, & \hbox{if\, $I_a\neq 0$;} \\
\frac{\overline{t_1}+t_2\bj}{4|t_1|^2}, & \hbox{if\, $I_a=0$.} \\
\end{array}
\right.$$
For $a=t_1+t_2\bj\in  Z(\bh_s)-\{0\}$, we have the following equations:
 \begin{equation}aa^+a=a,\  a^+aa^+=a^+,\  aa^+=\frac{1}{2}\big(1+\frac{t_2}{\overline{t_1}}\bj\big),
 \  a^+a=\frac{1}{2}\big(1+\frac{t_2}{t_1}\bj\big).\end{equation}

\begin{lem}(c.f.\cite[Corollary 3.1]{cao})\label{lemmoor}
	Let $a=t_1+t_2\bj\in  Z(\bh_s)-\{0\}$. Then the equation $ax=d$ is  solvable
	if and only if $aa^+d=\frac{1}{2}(1+\frac{t_2}{\overline{t_1}}\bj)d=d$,
 in which case all solutions are given by $$x=a^+d+(1-a^+a)y
 =\frac{\overline{t_1}+t_2\bj}{4|t_1|^2}d+\frac{1}{2}(1-\frac{t_2}{t_1}\bj)y,\forall y\in \bh_s.$$
 \end{lem}

\begin{pro}\label{proc=0}
	Let $f(x)=ax^2+c, a=1+a_2\bj+a_3\bk\in Z(\bh_s)$. Then $f(x)=0$ is solvable, then its companion  polynomial must be   $$C(x)=0.$$
\end{pro}
\begin{proof}
By (\ref{czpoly}),  the  companion  polynomial of  Equation I is $$C(x)=2P_{ac}x^2+I_c=0.$$
If $f(x)=0$ is solvable, then there exists $x$ such that $c=-ax^2$. Hence $$I_c=I_aI_x^2=0.$$    Because  $a=1+(a_2+a_3\bi)\bj\in Z(\bh_s)$, we have
\begin{equation}\label{ainvf}a^+=\frac{a}{4},\ \ aa^+=a^+a=\frac{a}{2}.\end{equation}
By Lemma \ref{lemmoor}, $ax^2+c=0$ is solvable if and only if
$aa^+c=c.$
That is  $$ac=2c.$$
Thus $\Re(ac)=2\Re(c)$. Hence $$P_{ac}=c_0-a_2c_2-a_3c_3=2\Re(c)-\Re(ac)=0.$$
This concludes the proof.
\end{proof}

 \begin{thm}\label{thm2.1}
    The quadratic equation  $ax^2+c=0$ is solvable if and only if $$ac=2c.$$
 If $ax^2+c=0$ is solvable then $$x=\sqrt[s]{-\frac{c}{2}+\frac{\bar{a}}{2}y},$$
  where $y\in \bh_s$ satisfying that $$\sqrt[s]{-\frac{c}{2}+\frac{\bar{a}}{2}y}\neq \emptyset.$$
   \end{thm}

\begin{proof}
By the proof of Proposition \ref{proc=0}, $ax^2+c=0$ is solvable if and only if $ac=2c.$
Under this condition,
$$x^2=a^+(-c)+(1-a^+a)y
=-\frac{c}{2}+(1-\frac{a}{2})y=-\frac{c}{2}+\frac{\bar{a}}{2}y,\forall y\in \bh_s.$$
It follows from   \cite[Theorem 2.1]{caoaxiom} that
$$x=\sqrt[s]{-\frac{c}{2}+\frac{\bar{a}}{2}y},$$
where $y\in \bh_s$ such that $\sqrt[s]{-\frac{c}{2}+\frac{\bar{a}}{2}y}\neq \emptyset.$
\end{proof}

\begin{exam}\label{exam4.1}
Consider the quadratic equation $(1+\bj)x^2 -1-\bj=0$.
That is, $a=1+\bj,c=-1-\bj.$
 $$x=\sqrt[s]{\frac{1}{2}(1+\bj)+\frac{1}{2}(1-\bj)y},$$
 where $y\in \bh_s$ such that $\sqrt[s]{\frac{1}{2}(1+\bj)+\frac{1}{2}(1-\bj)y}\neq \emptyset.$
%
%
%
\end{exam}

\section{Equation II for SZ}

We will find the necessary and sufficient conditions of  Equation II to have a solution such that  $2x_0+b\in Z(\bh_s)$.

For the sake of simplification, we define the following three numbers related to  Equation II:
\begin{equation}\label{t1t2d}\delta=a_2b_3-a_3b_2+b_1,\,\,t_1=c_2-c_0a_2-a_3c_1,\,\,t_2=c_3-c_0a_3+a_2c_1.\end{equation}
The above numbers  in fact are
\begin{equation}\label{t1t2d1}\delta=P_{ a \bi,b},\quad t_1=-P_{ a \bj,c},\quad t_2=-P_{\bk a,c}.\end{equation}

We relable the real nonlinear system (\ref{rsym2}) as follows.
\begin{eqnarray}
	x_0^2-x_1^2+x_2^2+x_3^2+2a_2x_0x_2+2a_3x_0x_3-b_1x_1+b_2x_2+b_3x_3+c_0=0,\label{ae1}\\
	2x_0x_1-2a_2x_0x_3+2a_3x_0x_2+b_1x_0-b_2x_3+b_3x_2+c_1=0,\label{ae2}\\
	2x_0x_2+a_2(x_0^2-x_1^2+x_2^2+x_3^2)+2a_3x_0x_1-b_1x_3+b_2x_0+b_3x_1+c_2=0,\label{ae3}\\
	2x_0x_3-2a_2x_0x_1+a_3(x_0^2-x_1^2+x_2^2+x_3^2)+b_1x_2-b_2x_1+b_3x_0+c_3=0.\label{ae4}
\end{eqnarray}

  Suppose $x=x_0+x_1\bi+x_2\bj+x_3\bk\in SZ$ is a solution of  Eq.(\ref{linearx1}).
By Eq.(\ref{linearx1}), we have
\begin{equation}\label{inpax0}\left\langle 2x_0a+b,2x_0a+b \right\rangle=4x_0P_{ab}+I_b=0.\end{equation}
Based on this, we  divide our consideration into  two cases: $$P_{ab}\neq 0\mbox{ and }P_{ab}=0.$$

  \subsection{Case $P_{ab}\neq 0$}
If $P_{ab}\neq 0$ then by (\ref{inpax0}) we have \begin{equation}\label{slx0}x_0=\frac{-I_b}{4P_{ab}}.\end{equation}
 We reformulate  Eqs.(\ref{ae1}) and (\ref{ae2}) as
\begin{eqnarray}
 -x_1^2+x_2^2+x_3^2-b_1x_1+(b_2+2a_2x_0)x_2+(b_3+2a_3x_0)x_3+c_0+x_0^2&=&0,\label{1aeq1}\\
               2x_0x_1+(b_3+2a_3x_0)x_2-(b_2+2a_2x_0)x_3&=&-c_1-b_1x_0.\label{1aeq2}
 \end{eqnarray}
   Using  Eq.(\ref{ae1})$\times a_2$+Eq.(\ref{ae2})$\times a_3- $ Eq.(\ref{ae3})
    and Eq.(\ref{ae1})$\times a_3-$ Eq.(\ref{ae2})$\times a_2-$ Eq.(\ref{ae4}), we obtain
\begin{eqnarray}
  (-a_2b_1-b_3)x_1+(a_2b_2+a_3b_3)x_2+(a_2b_3-a_3b_2+b_1)x_3&=&c_2-c_0a_2-a_3c_1+(b_2-a_3b_1)x_0,\label{1aeq3}\\
 (-a_3b_1+b_2)x_1+(a_3b_2-a_2b_3-b_1)x_2+(a_2b_2+a_3b_3)x_3&=& c_3-c_0a_3+a_2c_1+(b_3+a_2b_1)x_0.\label{1aeq4}
\end{eqnarray}
Let $y=(x_1,x_2,x_3)^T$.  Eqs.(\ref{1aeq2})-(\ref{1aeq4}) can be expressed as
  \begin{equation}\label{lineareqs5} Ay=u,\end{equation}
 where
    \begin{equation}\label{matrixA}A=\left(
  \begin{array}{ccc}
  2x_0&b_3+2a_3x_0&-b_2-2a_2x_0\\
       -a_2b_1-b_3 & a_2b_2+a_3b_3 &a_2b_3-a_3b_2+b_1\\
    -a_3b_1+b_2 & a_3b_2-a_2b_3-b_1 & a_2b_2+a_3b_3 \\
    \end{array}
\right)\end{equation} and  \begin{equation}\label{matrixu}u=\left(
                 \begin{array}{c}
                   -c_1-b_1x_0\\
                  t_1+(b_2-a_3b_1)x_0 \\
                  t_2+(b_3+a_2b_1)x_0\\
                   \end{array}
               \right).\end{equation}

\begin{pro}\label{prop5.2} Let $x_0=\frac{-I_b}{4P_{ab}}$ and $a_2^2+a_3^2=1$. Let  $A$ be given by (\ref{matrixA}). Then
 $$\det(A)=0.$$
\end{pro}

\begin{proof}
Let \begin{equation}\label{matrixB}B=\left(
  \begin{array}{ccc}
  2x_0&b_3&-b_2\\
       -a_2b_1-b_3 & a_2b_2+2a_3b_3+a_2a_3b_1 &-a_3b_2+a_3^2b_1 \\
    -a_3b_1+b_2 & -a_2b_3-a_2^2b_1& 2a_2b_2+a_3b_3-a_2a_3b_1 \\
    \end{array}
\right).\end{equation}
It is obvious that $B$ is obtained by performing  elementary  column transformations from  $A$.
 It can be verified that $\det(B)=0$. Therefore $\det(A)=\det(B)=0$.
\end{proof}
 Let \begin{equation}\label{matrixM}M=\left(\begin{array}{cc}
        a_2b_2+a_3b_3 &a_2b_3-a_3b_2+b_1 \\
    a_3b_2-a_2b_3-b_1& a_2b_2+a_3b_3\\
  \end{array}\right)=\left(\begin{array}{cc}
        -P_{ab} &\delta \\
   -\delta& -P_{ab}\\
  \end{array}\right).\end{equation} Since $P_{ab}\neq 0$, the subdeterminant
  $$m=:\det(M)=P_{ab}^2+\delta^2>0.$$
    By Proposition \ref{prop5.2}, this means that  $rank(A)=2$. We reformulate  Eqs.(\ref{1aeq3}) and (\ref{1aeq4}) as
   \begin{equation}Mz=v,\end{equation}
  where $$z=(x_2,x_3)^T,v=\left(
                 \begin{array}{c}
                   t_1+(b_2-a_3b_1)x_0+(a_2b_1+b_3)x_1 \\
                  t_2+(a_2b_1+b_3)x_0+(a_3b_1-b_2)x_1\\
                   \end{array}
                 \right).$$

  Let
$$k_1:=-P_{ab}(a_2b_1+b_3)-\delta(a_3b_1-b_2)=2b_2\delta-a_3I_b,$$
$$k_2:=-P_{ab}(b_2-a_3b_1)-\delta(a_2b_1+b_3)=-2b_3\delta-a_2I_b$$
and
\begin{equation*}\Delta_1=\frac{-P_{ab}t_1-\delta t_2}{m},\Delta_2=\frac{\delta t_1-P_{ab}t_2}{m}.\end{equation*}
Note that
$$ m=P_{ab}^2+\delta^2=b_1^2+b_2^2+b_3^2+2a_2b_1b_3-2a_3b_1b_2=2b_1\delta-I_b$$
and $$k_1^2+k_2^2=m^2.$$
   Because
  $$M^{-1}=\frac{1}{m}\left(\begin{array}{cc}
       -P_{ab} & -\delta\\
   \delta&-P_{ab}\\
  \end{array}\right)\mbox{  and } z=M^{-1}v,$$
    we have
\begin{eqnarray}x_2&=&\frac{-P_{ab}[t_1+(b_2-a_3b_1)x_0+(a_2b_1+b_3)x_1]
-\delta[t_2+(a_2b_1+b_3)x_0+(a_3b_1-b_2)x_1]}{m}\nonumber\\
&=&\frac{-P_{ab}(a_2b_1+b_3)-\delta(a_3b_1-b_2)}{m}x_1+\frac{-P_{ab}(b_2-a_3b_1)-\delta(a_2b_1+b_3)}{m}x_0
+\frac{-P_{ab}t_1-\delta t_2}{m}\nonumber\\
&=&\frac{k_1}{m}x_1+\frac{k_2}{m}x_0+\Delta_1\label{x2s5}\end{eqnarray}
and
\begin{eqnarray}x_3&=&\frac{\delta[t_1+(b_2-a_3b_1)x_0+(a_2b_1+b_3)x_1]-P_{ab}[t_2+(a_2b_1+b_3)x_0+(a_3b_1-b_2)x_1]}{m}\nonumber\\
&=&\frac{\delta(a_2b_1+b_3)-P_{ab}(a_3b_1-b_2)}{m}x_1
+\frac{-P_{ab}(a_2b_1+b_3)-\delta(a_3b_1-b_2)}{m}x_0+\frac{\delta t_1-P_{ab}t_2}{m}\nonumber\\
&=&-\frac{k_2}{m}x_1+\frac{k_1}{m}x_0+\Delta_2.\label{x3s5}\end{eqnarray}
Substituting the above two formulas  in Eq.(\ref{1aeq2}), we  have
\begin{eqnarray*}\Big(2x_0+\frac{b_3k_1+b_2k_2+2a_3k_1x_0+2a_2k_2x_0}{m}\Big)x_1+F=0,
\end{eqnarray*}
where \begin{eqnarray}\label{condthm5.1}F=\frac{2a_3k_2-2a_2k_1}{m}x_0^2
+\Big(\frac{b_3k_2-b_2k_1}{m}+2a_3\Delta_1-2a_2\Delta_2+b_1\Big)x_0+b_3\Delta_1-b_2\Delta_2+c_1.\end{eqnarray}
Note that $$2x_0+\frac{b_3k_1+b_2k_2+2a_3k_1x_0+2a_2k_2x_0}{m}=0.$$ By the solvability of $Ay=u$, we should have $F=0$.
We remark that the fact that the coefficient of $x_1$ is zero
 is guaranteed by $\det(A)=0$ and $F=0$ is just a restatement of $rank(A)=rank(A,u)=2$.

Substituting  $x_2$ and $x_3$ of  (\ref{x2s5}) and (\ref{x3s5}) in Eq.(\ref{1aeq1}), we obtain
$$Rx_1+L=0,$$
where \begin{eqnarray}R=\frac{2k_1\Delta_1-2k_2\Delta_2+b_2k_1-b_3k_2+2(a_2k_1-a_3k_2)x_0-mb_1}{m}\end{eqnarray} and
\begin{eqnarray}L&=&b_2\Delta_1+b_3\Delta_2+\Delta_1^2+\Delta_2^2+c_0+\frac{2(a_2k_2+a_3k_1+m)}{m}x_0^2\nonumber\\
&&+\frac{(2k_2\Delta_1+2k_1\Delta_2+b_2k_2+b_3k_1+2a_2\Delta_1m+2a_3\Delta_2m)}{m}x_0.\end{eqnarray}
 If $R=0$  we should have $L=0$ and in this case, $x_1$ is arbitrary.
 If $R\neq 0$ then $$x_1=\frac{-L}{R}.$$

Summarizing our reasoning process, we figure out the following  conditions.

 \begin{defi}\label{def5.1}
For the coefficients $a,b,c$ in Equation II  such that $P_{ab}\neq 0,$
we set \begin{equation}\label{x0s5thm}x_0=\frac{-I_b}{4P_{ab}},\end{equation}

\begin{equation}\label{k1k2mthm}k_1=2b_2\delta-a_3I_b,k_2=-2b_3\delta-a_2I_b, m=2b_1\delta-I_b,\end{equation}
\begin{equation}\label{Delta}\Delta_1=\frac{-P_{ab}t_1-\delta t_2}{m},\Delta_2=\frac{\delta t_1-P_{ab}t_2}{m},\end{equation}
 \begin{eqnarray}\label{thmR}R=\frac{2k_1\Delta_1-2k_2\Delta_2+b_2k_1-b_3k_2+2(a_2k_1-a_3k_2)x_0-mb_1}{m},\end{eqnarray}
   \begin{eqnarray}L&=&b_2\Delta_1+b_3\Delta_2+\Delta_1^2+\Delta_2^2+c_0+\frac{2(a_2k_2+a_3k_1+m)}{m}x_0^2\nonumber\\
&&+\frac{(2k_2\Delta_1+2k_1\Delta_2+b_2k_2+b_3k_1+2a_2\Delta_1m+2a_3\Delta_2m)}{m}x_0,\label{thmL}\end{eqnarray}
and
\begin{eqnarray}\label{condthm3.2F}F=\frac{2a_3k_2-2a_2k_1}{m}x_0^2
+\Big(\frac{b_3k_2-b_2k_1}{m}+2a_3\Delta_1-2a_2\Delta_2+b_1\Big)x_0+b_3\Delta_1-b_2\Delta_2+c_1.\end{eqnarray}
We say  $(a,b,c)$   satisfies  {\bf Condition 1} if the following two conditions hold:
\begin{itemize}
		\item [(1)]  $F=0;$
\item [(2)] If $R=0$ then $L=0$.
\end{itemize}
 \end{defi}

Summarizing the previous results, we obtain the following theorem.
\begin{thm}\label{thm5.1}
 Equation II with $P_{ab}\neq 0$  has a  solution $x\in SZ$ if and only if Condition 1 holds.
 Let $x_0,k_1,k_2,m,\Delta_1,\Delta_2,R,L,F$ be given by Definition \ref{def5.1}.
If Condition 1 holds, then we have the following cases:
\begin{itemize}
	\item [(1)]
 If $R\neq 0$  then Equation II has a solution:
 $$x=x_0-\frac{L}{R}\bi+x_2\bj+x_3\bk,$$
where $$x_2= -\frac{k_1L}{mR}+\frac{k_2}{m}x_0+\Delta_1$$
and
 $$x_3=\frac{k_2L}{mR}+\frac{k_1}{m}x_0+\Delta_2.$$

\item [(2)]  If $R=0$ then Equation II has solutions:
$$x=x_0+x_1\bi+(\frac{k_1}{m}x_1+\frac{k_2}{m}x_0+\Delta_1)\bj+(-\frac{k_2}{m}x_1+\frac{k_1}{m}x_0+\Delta_2)\bk,
 \forall x_1\in \br.$$
\end{itemize}
\end{thm}

\begin{exam}\label{exam5.1}
	Consider the quadratic equation $(1+\bj)x^2 +(\bi+2\bj+\bk)x-\frac{1}{4}+\frac{5}{2}\bi+\frac{3}{4}\bi+\frac{5}{2}\bk=
0$.
 That is, $a=1+\bj,b=\bi+2\bj+\bk$ and $c=-\frac{1}{4}+\frac{5}{2}\bi+\frac{3}{4}\bi+\frac{5}{2}\bk$. In this case
$$x_0=-\frac{1}{2},k_1=8,k_2=0,\Delta_1=-1,\Delta_2=\frac{3}{2},m=8,R=-2,L=2, F=0.$$
Therefore $(a,b,c)$ satisfies  {\it Condition 1} and $x_1=-\frac{L}{R}=1, x_2=0,x_3=1$.
 Thus $$x=-\frac{1}{2}+\bi+\bk$$ is a solution of the given quadratic equation.
\end{exam}

\begin{exam}\label{exam5.2}
	Consider the quadratic equation $(1+\bj)x^2 +(\bi+\bj)x-1+\bi= 0$.
That is, $a=1+\bj,b=\bi+\bj$ and $c=-1+\bi$. In this case $$x_0=0,k_1=2,k_2=0,\Delta_1=0,\Delta_2=1,m=2,R=L=0,F=0.$$
Therefore $(a,b,c)$ satisfies  {\it Condition 1}. In this case $x_1$ is arbitrary, $x_2=x_1,x_3=1$.
Thus $$x=x_1\bi+x_1\bj+\bk,\forall x_1\in \br$$  are solutions of   the given quadratic equation.
\end{exam}

\subsection{Case $P_{ab}=0$}
In this subsection, we will find the necessary and sufficient conditions of  Equation II with $P_{ab}=0$ to have  a solution $x\in SZ$.

We begin with a proposition, which describes the linear relation of $x_i,i=0,\cdots, 3$.
\begin{pro}\label{prop5.1}
	Suppose that $P_{ab}=0$. Then the solution $x$ of $ax^2+bx+c=0$ satisfies the following linear equation:
	\begin{equation}\label{lineareqn} Ay=u,\end{equation}
	where  $y=(x_0,x_1,x_2,x_3)^T$,  \begin{equation}\label{matrixA1n}A=\left(
		\begin{array}{cccc}
			b_2-a_3b_1&a_2b_1+b_3&0&-\delta\\
			a_2b_1+b_3&a_3b_1-b_2&\delta&0
		\end{array}
		\right),u=\left(
		\begin{array}{c}
			-t_1\\
			-t_2
		\end{array}
		\right)\end{equation}
and $t_1,t_2,\delta$ are given by (\ref{t1t2d1}).	
	
\end{pro}
\begin{proof}
	Note that $a_2^2+a_3^2=1$ and $a_2b_2+a_3b_3=0$.
	Using Eq.(\ref{ae3})$-$Eq.(\ref{ae1})$\times a_2-$Eq.(\ref{ae2})$\times a_3$, we have
	\begin{equation}\label{new3}(b_2-a_3b_1)x_0+(a_2b_1+b_3)x_1+(a_3b_2-a_2b_3-b_1)x_3+c_2-a_2c_0-a_3c_1=0.\end{equation}
	Using Eq.(\ref{ae4})$-$Eq.(\ref{ae1})$\times a_3+$ Eq.(\ref{ae2})$\times a_2$, we have
	\begin{equation}\label{new4}(a_2b_1+b_3)x_0+(a_3b_1-b_2)x_1+(a_2b_3-a_3b_2+b_1)x_2+c_3+a_2c_1-a_3c_0=0.\end{equation}
	This completes the proof.
\end{proof}

Suppose that Equation II with $P_{ab}=0$  has a solution $x\in SZ$.
By Proposition \ref{prop5.1}, under the condition  $P_{ab}=0$, we have
      \begin{eqnarray}
 (b_2-a_3b_1)x_0+(a_2b_1+b_3)x_1+(a_3b_2-a_2b_3-b_1)x_3+t_1=0,\label{2new3}\\
 (a_2b_1+b_3)x_0+(a_3b_1-b_2)x_1+(a_2b_3-a_3b_2+b_1)x_2+t_2=0.\label{2new4}
\end{eqnarray}
 Since
$$\left\langle 2x_0a+b,2x_0a+b \right\rangle=4x_0P_{ab}+I_b=0.$$  By our assumption $P_{ab}=0$, we must have $I_b=0$.
 By $P_{ab}=I_b=0$, we have $$b_1^2-(a_3b_2-a_2b_3)^2=0.$$
 Thus we have $\delta=2b_1$ or $\delta=0$.
  We divide our consider into two subcases: $$\delta=2b_1\mbox{ and }\delta=0.$$
We mention that $\delta=2b_1$ is $P_{\bi a ,b}=0$ and $\delta=0$ is $P_{a \bi ,b}=0$
\subsubsection{Subcase $P_{\bi a ,b}=0$}
   We begin with the case $\delta=2b_1$, that is $b_1=-a_3b_2+a_2b_3$.

If $b_1=-a_3b_2+a_2b_3$, then by $P_{ab}=0$ and $I_a=0$ we have $b_2-a_3b_1=2b_2,a_2b_1+b_3=2b_3.$
 Thus \begin{equation}\label{b3b2c1}b_3=a_2b_1,b_2=-a_3b_1.\end{equation} By our assumption $b\neq 0$, we have $b_1\neq 0$.
  Hence Eqs. (\ref{2new3}) and (\ref{2new4}) become
     \begin{eqnarray*}
  2b_2x_0+2b_3x_1-2b_1x_3+t_1=0,\\
 2b_3x_0-2b_2x_1+2b_1x_2+t_2=0.
\end{eqnarray*}
  From the above and Eq.(\ref{b3b2c1}), we get
      \begin{eqnarray}
   x_2&=&-a_2x_0-a_3x_1-\frac{t_2}{2b_1},\label{s5newx2}\\
  x_3&=&-a_3x_0+a_2x_1+\frac{t_1}{2b_1}.\label{s5newx3}
\end{eqnarray}
Substituting   the above two formulas of $x_2$ and $x_3$ in Eq.(\ref{ae2}),
  we obtain
  \begin{eqnarray*}
-\frac{a_2t_1+a_3t_2}{b_1}x_0-\frac{a_2t_2-a_3t_1}{2}+c_1=0.
   \end{eqnarray*}
  If $a_2t_1+a_3t_2=0$ then we must have  $-\frac{a_2t_2-a_3t_1}{2}+c_1=0$ and in this case  $x_0$ is arbitrary.
If $a_2t_1+a_3t_2\neq 0$ then $$x_0=\frac{(a_3t_1-a_2t_2+2c_1)b_1}{2(a_2t_1+a_3t_2)}.$$

  Substituting $x_2$ and $x_3$ of  (\ref{s5newx2}) and (\ref{s5newx3}) in Eq.(\ref{ae1}),
    we obtain
\begin{eqnarray*}
\frac{a_2t_1+a_3t_2}{b_1}x_1+\frac{t_1^2+t_2^2}{4b_1^2}+\frac{a_2t_1+a_3t_2}{2}+c_0=0.
   \end{eqnarray*}
If $a_2t_1+a_3t_2=0$ then we need $\frac{t_1^2+t_2^2}{4b_1^2}+c_0=0$ and $x_1$ is arbitrary.
If $a_2t_1+a_3t_2\neq 0$ then $$x_1=-\frac{t_1^2+t_2^2+2b_1^2(a_2t_1+a_3t_2)+4b_1^2c_0}{4b_1(a_2t_1+a_3t_2)}.$$

By the above reasoning process, we figure out the following condition.

 \begin{defi}\label{def5.2}
  For the coefficients $a,b,c$ in Equation II  such that $P_{ab}=0$ and $I_b=0$,
        $(a,b,c)$   satisfies  {\bf Condition 2} if the following two conditions hold:
\begin{itemize}
		\item [(1)]  $P_{\bi a,b}=0;$
\item [(2)]  If $a_2t_1+a_3t_2=0$ then  $a_2t_2-a_3t_1-2c_1=0$ and  $t_1^2+t_2^2+4b_1^2c_0=0$.
\end{itemize}
 \end{defi}

Summarizing the previous results, we obtain the following theorem.
\begin{thm}\label{thm3.2}
 Equation II with $P_{ab}=0$ and $P_{\bi a,b}=0$ has a  solution $x\in SZ$  if and only if Condition 2 holds.
If Condition 2 holds, then we have the following cases.
	\begin{itemize} \item [(1)] If  $a_2t_1+a_3t_2\neq 0$ then   Equation II  has a solution
 $$x=x_0+x_1\bi+x_2\bj+x_3\bk,$$
where
\begin{eqnarray*}\label{s5x0}x_0=\frac{(a_3t_1-a_2t_2+2c_1)b_1}{2(a_2t_1+a_3t_2)}
   \end{eqnarray*}
 \begin{eqnarray*}\label{s5x1}x_1=-\frac{t_1^2+t_2^2+2b_1^2(a_2t_1+a_3t_2)+4b_1^2c_0}{4b_1(a_2t_1+a_3t_2)},
   \end{eqnarray*}
and
  \begin{eqnarray*}
  x_2&=&-a_2x_0-a_3x_1-\frac{t_2}{2b_1},\\
  x_3&=&-a_3x_0+a_2x_1+\frac{t_1}{2b_1}.
\end{eqnarray*}
 \item [(2)]If  $a_2t_1+a_3t_2=0$
 then   Equation II  has solutions
 $$x=x_0+x_1\bi+x_2\bj+x_3\bk,\forall x_0,x_1\in \br,$$
where
  \begin{eqnarray*}
   x_2&=&-a_2x_0-a_3x_1-\frac{t_2}{2b_1},\\
  x_3&=&-a_3x_0+a_2x_1+\frac{t_1}{2b_1}.
\end{eqnarray*}
\end{itemize}
\end{thm}

\begin{exam}\label{examthm5.3}
	Consider the quadratic equation $(1+\bj)x^2 +(\bi+\bk)x+1-\bi=0$.
That is, $a=1+\bj,b=\bi+\bk$ and $c=1-\bi$. In this case $t_1=t_2=-1$ and $a_2t_1+a_3t_2=-1$.
The equation $ax^2+bx+c=0$ has a solution
$x=\frac{1}{2}+\bi+\frac{1}{2}\bk.$
\end{exam}

\begin{exam}\label{examthm5.4}
	Consider the quadratic equation $(1+\bj)x^2 +(\bi+\bk)x-1+\bi-\bj+\bk=0$.
That is, $a=1+\bj,b=\bi+\bk$ and $c=-1+\bi-\bj+\bk$.
In this case $t_1=0,t_2=2$ and $a_2t_1+a_3t_2=0$.  The equation $ax^2+bx+c=0$ has  solutions
$x=x_0+x_1\bi-(1+x_0)\bj+x_1\bk,\forall x_0,x_1\in \br.$
\end{exam}
\subsubsection{Subcase $P_{ a \bi,b}=0$}
We now consider the  second case $\delta=0$, that is, $b_1=a_3b_2-a_2b_3$.  If $b_1=a_3b_2-a_2b_3$ then by $P_{ab}=0$ and
$I_a=0$
we have $$b_2-a_3b_1=a_2(a_2b_2+a_3b_3)=0,\,a_2b_1+b_3=a_3(a_2b_2+a_3b_3)=0.$$
So we have
\begin{equation}\label{b3b2c2}b_2=a_3b_1,b_3=-a_2b_1.\end{equation}
Hence we have
\begin{equation}\label{ab1i}b=ab_1\bi\end{equation}

From the above formulas, Eqs.(\ref{2new3}) and (\ref{2new4}) imply that
\begin{equation}\label{ac1}c_2-a_2c_0-a_3c_1=0,\,c_3+a_2c_1-a_3c_0=0.\end{equation}
By $I_a=0$ and the above two conditions, we have
\begin{equation}\label{ac2}P_{ac}=\left\langle a,c\right\rangle=c_0-a_2c_2-a_3c_3=0.\end{equation}
From this, we get
\begin{equation}\label{ac3}c_1=a_3c_2-a_2c_3.\end{equation}
Eqs. (\ref{ac1})-(\ref{ac3}) are equivalent to the condition \begin{equation}\label{acend}ac=2c.\end{equation}
Under the condition $P_{ab}=I_a=I_b=0,b_2=a_3b_1,b_3=-a_2b_1$ and $ac=2c$, we have
\begin{center}Eq.(\ref{ae3})$=$ Eq.(\ref{ae1})$\times a_2+$Eq.(\ref{ae2})$\times a_3$ and
Eq.(\ref{ae4})$=$ Eq.(\ref{ae1})$\times a_3-$Eq.(\ref{ae2})$\times a_2$.\end{center}
Hence in this case Equation II  only has two independent equalities Eqs.(\ref{ae1}) and (\ref{ae2}), which can be reformulated
as
\begin{eqnarray}
x_0^2+2(a_2x_2+a_3x_3)x_0-x_1^2+x_2^2+x_3^2-b_1x_1+a_3b_1x_2-a_2b_1x_3+c_0=0,\label{ab0eq1}\\
x_0(2x_1-2a_2x_3+2a_3x_2+b_1)=b_1(a_2x_2+a_3x_3)-c_1.\label{ab0eq2}
\end{eqnarray}
  These are underdetermined system of equations.

Before going on, we make the following remark.
\begin{rem}
Note that $$ax^2+bx+c=ax^2+ab_1\bi x+\frac{ac}{2}=a(x^2+b_1\bi x+\frac{c}{2}).$$
We reformulate the Equation II
$$a(x^2+b_1\bi x+\frac{c}{2})=0.$$
By Lemma \ref{lemmoor}, the above equation is equivalent to
$$x^2+b_1\bi x+\frac{c}{2}=(1-a^+a)y=\frac{\bar{a}}{2}y=0, \forall y\in \bh_s.$$
For simplification, we express  it as
$$x^2+b_1\bi x+\frac{c}{2}+\bar{a}y=0, \forall y\in \bh_s.$$
In principle, we can solve the above equation for a specific $y$ by the approaches in \cite{caoaxiom,Ireneu,Irene,Opfer18,Opfer17}.
\end{rem}
To avoid introducing the arbitrary $y$, we choose another method as follows.

Note that $b_1\neq 0$.
By  Eq.(\ref{ab0eq2}), if $x_0=0$ then  $$a_2x_2+a_3x_3=\frac{c_1}{b_1}.$$
we treat now the cases $a_2=0$ and $a_2\neq 0$, respectively.

If $a_2=0$ then $a_3\neq 0$ and therefore $$x_3=\frac{c_1}{a_3b_1}.$$ Note that $a_3^2=1$.
Substituting $x_0=0$ and $x_3=\frac{c_1}{a_3b_1}$ in Eq.(\ref{ab0eq1}), we obtain
$$x_2^2+a_3b_1x_2+\frac{c_1^2}{b_1^2}+c_0-x_1^2-b_1x_1=0.$$
So we have a solution $$x=x_1\bi+x_2\bj+\frac{c_1}{a_3b_1}\bk,$$
where $$x_2=\frac{-a_3b_1\pm \sqrt{b_1^2-4(\frac{c_1^2}{b_1^2}+c_0-x_1^2-b_1x_1)}}{2}$$
and $x_1\in \br$ satisfies
$$\frac{c_1^2}{b_1^2}+c_0-\frac{b_1^2}{4}-x_1^2-b_1x_1\leq 0.$$

If $a_2\neq 0$ then \begin{equation}\label{addx2}x_2=\frac{c_1}{a_2b_1}-\frac{a_3}{a_2}x_3.\end{equation}
Substituting $x_0=0$ and the above formula 
in Eq.(\ref{ab0eq1}), we obtain
$$x_1^2+b_1x_1+t=0,$$
where $$t=-\frac{1}{a_2^2}x_3^2+(\frac{2a_3c_1+a_2b_1^2}{a_2^2b_1})x_3-(c_0+\frac{a_3c_1}{a_2}+\frac{c_1^2}{a_2^2b_1^2}).$$
Hence $x_1$ can be expressed by $x_3$ as
\begin{equation}\label{addx1}x_1=\frac{-b_1\pm \sqrt{b_1^2-4t}}{2}\end{equation}
and $x_3\in \br$ satisfies
\begin{equation}\label{addb1}b_1^2-4t=\frac{4}{a_2^2}[x_3^2-(\frac{2a_3c_1+a_2b_1^2}{b_1})x_3+\frac{4(a_2^2b_1^2c_0+a_2b_1^2a_3c_1+c_1^2)+b_1^4a_2^2}{4b_1^2}]
\geq 0.\end{equation}
So we have  solutions $$x=x_1\bi+x_2\bj+x_3\bk,$$
where $x_1,x_2$ are given by (\ref{addx1}) and (\ref{addx2}), and $x_3\in \br$ satisfies (\ref{addb1}).

If $x_0\neq 0$ then from  Eq.(\ref{ab0eq2}) we have
 $$2x_1-2a_2x_3+2a_3x_2+b_1=\frac{a_3b_1x_3+a_2b_1x_2-c_1}{x_0}.$$
From this, we get $$x_1=(\frac{a_2b_1}{2x_0}-a_3)x_2+(\frac{a_3b_1}{2x_0}+a_2)x_3-\frac{c_1+b_1x_0}{2x_0}.$$
Substituting the above formula in  Eq.(\ref{ab0eq1}) and rearranging the equation, we obtain
\begin{eqnarray*}
&&x_0^4+2(a_2x_2+a_3x_3)x_0^3+[(a_2x_2+a_3x_3)^2+b_1(a_3x_2-a_2x_3)+c_0+\frac{b_1^2}{4}]x_0^2\\
&&+[a_2a_3b_1(x_2^2-x_3^2)+b_1(a_3^2-a_2^2)x_2x_3+c_1(a_2x_3-a_3x_2)]x_0-\frac{[b_1(a_2x_2+a_3x_3)-c_1]^2}{4}=0.
\end{eqnarray*}
Let \begin{eqnarray*}
f(z)&=&z^4+2(a_2x_2+a_3x_3)z^3+[(a_2x_2+a_3x_3)^2+b_1(a_3x_2-a_2x_3)+c_0+\frac{b_1^2}{4}]z^2\\
&&+[a_2a_3b_1(x_2^2-x_3^2)+b_1(a_3^2-a_2^2)x_2x_3+c_1(a_2x_3-a_3x_2)]z-\frac{[b_1(a_2x_2+a_3x_3)-c_1]^2}{4}.
\end{eqnarray*}
Then $$f(0)=-\frac{[b_1(a_2x_2+a_3x_3)-c_1]^2}{4}\leq 0,\lim_{z\to +\infty} f(z)=+\infty,\lim_{z\to -\infty} f(z)=+\infty.$$

If $a_2x_2+a_3x_3\neq \frac{c_1}{b_1}$ then  $f(0)<0$,  and $f(z)=0$  has at least two real  solutions  $z_1\in (-\infty,0)$
and $z_2\in (0,\infty)$.
 Let $T\in \br$ be a solution of $f(z)=0$ with $a_2x_2+a_3x_3\neq \frac{c_1}{b_1}$.
 Then Equation II has a solution
 $$x=T+x_1\bi+x_2\bj+x_3\bk,$$
  where
$$x_1=(\frac{a_2b_1}{2T}-a_3)x_2+(\frac{a_3b_1}{2T}+a_2)x_3-\frac{c_1+b_1T}{2T}.$$

If \begin{equation}\label{ab0eqx23}a_2x_2+a_3x_3=\frac{c_1}{b_1},\end{equation} then  by (\ref{ab0eq2}) and our assumption
$x_0\neq 0$
 we have
\begin{eqnarray}
               2x_1-2a_2x_3+2a_3x_2+b_1=0.\label{ab0eq3}
  \end{eqnarray}
By (\ref{ab0eqx23}) and (\ref{ab0eq3}), we obtain that
\begin{eqnarray*}
 x_2=\frac{a_2c_1}{b_1}-\frac{a_3b_1}{2}-a_3x_1,\label{ab0eq4}\\
 x_3=\frac{a_3c_1}{b_1}+\frac{a_2b_1}{2}+a_2x_1.\label{ab0eq5}
  \end{eqnarray*}
Substituting the above formulas in (\ref{ab0eq1}), we obtain
\begin{eqnarray}
             x_0^2+\frac{2c_1}{b_1}x_0-b_1x_1+\frac{c_1^2}{b_1^2}-\frac{b_1^2}{4}+c_0=0.
  \end{eqnarray}
Hence
\begin{eqnarray*}
            x_1= \frac{1}{b_1}x_0^2+\frac{2c_1}{b_1^2}x_0+\frac{c_1^2}{b_1^3}-\frac{b_1}{4}+\frac{c_0}{b_1}.\label{s5newx1}
  \end{eqnarray*}
Form the above description,  Equation II has  solutions   $$x=x_0+x_1\bi+x_2\bj+x_3\bk,\forall x_0\neq 0,$$
where $x_1,x_2,x_3$ are expressed  by formulas containing $x_0$ as above.

Summarizing the previous results, we obtain the following theorem.

\begin{thm}\label{thm5.3}
 Equation II with $P_{ab}=0$ and $P_{a \bi,b}=0$  has a  solution $x\in SZ$  if and only if
 $ac=2c$.
If  Equation II is solvable then we have the following cases:
\begin{itemize}
 \item [(1)] Case $x_0=0$:
\begin{itemize}
  \item [(1.1)] if  $a_2=0$ then   Equation II has  solutions:$$x=x_1\bi+x_2\bj+\frac{c_1}{a_3b_1}\bk,$$
      where $$x_2=\frac{-a_3b_1\pm \sqrt{b_1^2-4(\frac{c_1^2}{b_1^2}+c_0-x_1^2-b_1x_1)}}{2}$$
      and $x_1$ is real numbers satisfies
      $$x_1^2+b_1x_1+\frac{b_1^2}{4}-\frac{c_1^2}{b_1^2}-c_0\geq 0.$$

  \item [(1.2)] if   $a_2\neq 0$ then   Equation II has  solutions:$$x=x_1\bi+x_2\bj+x_3\bk,$$
      where $x_3\in \br$  satisfies $$w=x_3^2-(\frac{2a_3c_1+a_2b_1^2}{b_1})x_3
      +\frac{4(a_2^2b_1^2c_0+a_2b_1^2a_3c_1+c_1^2)+b_1^4a_2^2}{4b_1^2}\ge 0$$
      and
      $$x_1=\frac{-b_1}{2}\pm \frac{\sqrt{w}}{a_2},$$
      $$x_2=\frac{c_1}{a_2b_1}-\frac{a_3}{a_2}x_3.$$
     \end{itemize}
\item [(2)] Case $x_0\neq 0$:
\begin{itemize}
  \item [(2.1)] $a_2x_2+a_3x_3\neq \frac{c_1}{b_1}$:
  Equation II has  solutions:
             $$x=T+x_1\bi+x_2\bj+x_3\bk,$$  where
            $T$ be a real solution  of the following equation:
             \begin{eqnarray*}
 z^4+2(a_2x_2+a_3x_3)z^3+[(a_2x_2+a_3x_3)^2+b_1(a_3x_2-a_2x_3)+c_0+\frac{b_1^2}{4}]z^2\\
 +[a_2a_3b_1(x_2^2-x_3^2)+b_1(a_3^2-a_2^2)x_2x_3+c_1(a_2x_3-a_3x_2)]z-\frac{[b_1(a_2x_2+a_3x_3)-c_1]^2}{4}=0
   \end{eqnarray*}
   and
$$x_1=(\frac{a_2b_1}{2T}-a_3)x_2+(\frac{a_3b_1}{2T}+a_2)x_3-\frac{c_1+b_1T}{2T}.$$

\item [(2.2)] $a_2x_2+a_3x_3=\frac{c_1}{b_1}$:
  Equation II has  solutions:
             $$x=x_0+x_1\bi+x_2\bj+x_3\bk,\forall x_0\neq 0,$$  where
     \begin{eqnarray*}
            x_1&=& \frac{1}{b_1}x_0^2+\frac{2c_1}{b_1^2}x_0+\frac{c_1^2}{b_1^3}-\frac{b_1}{4}+\frac{c_0}{b_1},\\
  x_2&=&\frac{a_2c_1}{b_1}-\frac{a_3b_1}{2}-a_3x_1,\\
 x_3&=&\frac{a_3c_1}{b_1}+\frac{a_2b_1}{2}+a_2x_1.
  \end{eqnarray*}
  \end{itemize}
  \end{itemize}
\end{thm}

\begin{exam}\label{exam5.5}
	Consider the quadratic equation \begin{equation}\label{exam5.5eq}(1+\bk)x^2 +(\bi+\bj)x+1+2\bi+2\bj+\bk= 0.\end{equation}
That is, $a=1+\bk,b=\bi+\bj$ and $c=1+2\bi+2\bj+\bk$.  Then we have the following cases:
\begin{itemize}
  \item[(1.1)]
  Eq.(\ref{exam5.5eq}) has  the following solutions
  $$x=x_1\bi-\Big(\frac{1}{2}\pm \sqrt{x_1^2+x_1-\frac{19}{4}}\Big)\bj+2\bk,$$
  where $x_1$ is arbitrary but satisfies $x_1^2+x_1-\frac{19}{4}\geq 0$.
  \item[(2.1)]  Eq.(\ref{exam5.5eq}) has  the following solutions

            $$x=T+x_1\bi+x_2\bj+x_3\bk,\forall x_3\neq 2, x_2 \in \br$$  where
            $T$ be a real solution  of the following equation:  \begin{eqnarray}\label{examp5.5x0}
 z^4+2x_3z^3+(x_3^2+x_2+\frac{5}{4})z^2+(x_2x_3-2x_2)z-\frac{(x_3-2)^2}{4}=0
   \end{eqnarray}
   and
$$x_1=-x_2+\frac{1}{2T}x_3-\frac{2+T}{2T}.$$
For example, if we take $x_2=x_3=1$, then Eq. (\ref{examp5.5x0}) has real solution $T_1=0.3914$ and $T_2=-0.1675$.
So we have solutions
$$x=0.3914-2.7773\bi+\bj+\bk,\mbox{  and  } x=-0.1675+1.4857\bi+\bj+\bk.$$
\item[(2.2)]  When $x_3=2$,  Eq.(\ref{exam5.5eq}) has  the following solutions
             $$x=x_0+x_1\bi-(x_1+\frac{1}{2})\bj+2\bk,\forall x_0\neq 0,$$  where
     \begin{eqnarray*}
            x_1= x_0^2+4x_0+\frac{19}{4}.
  \end{eqnarray*}
  \end{itemize}
\end{exam}

\begin{exam}\label{exam5.6}
	Consider the quadratic equation \begin{equation}\label{exam5.6eq}(1+\bj)x^2 +(-\bi+\bk)x-1+\bi-\bj-\bk= 0.\end{equation}
That is, $a=1+\bj,b=-\bi+\bk$ and $c=-1+\bi-\bj-\bk$.  Then we have the following cases.
\begin{itemize}
  \item[(1.2)] Eq.(\ref{exam5.6eq}) has  the following solutions
  $$x=(1+x_3)\bi-\bj+x_3\bk\mbox{ and  } x=-x_3\bi-\bj+x_3\bk,  \forall x_3\in \br.$$
  \item[(2.1)] Eq.(\ref{exam5.6eq}) has  the following solutions
            $$x=T+x_1\bi+x_2\bj+x_3\bk,\forall x_2\neq -1,x_3\in \br$$  where
            $T$ be a real solution  of the following equation:  \begin{eqnarray}\label{examp5.6x0}
 z^4+2x_2z^3+(x_2^2+x_3-\frac{3}{4})z^2+(x_2x_3+x_3)z-\frac{(x_2+1)^2}{4}=0
   \end{eqnarray}
   and
$$x_1=-\frac{1}{2T}x_2+x_3+\frac{T-1}{2T}.$$
For example, if we take $x_2=x_3=1$, then Eq. (\ref{examp5.6x0}) has real solution $x_0=-2$ and $x_0=0.362$.
So we have solutions
$$x=-2+2\bi+\bj+\bk,\mbox{  and  } x=0.3620-1.2621\bi+\bj+\bk.$$
\item[(2.2)]  When $x_2=-1$, Eq.(\ref{exam5.6eq}) has  the following solutions
             $$x=x_0+x_1\bi-\bj+(x_1-\frac{1}{2})\bk,\forall x_0\neq 0,$$  where
     \begin{eqnarray*}
            x_1= -x_0^2+2x_0+\frac{1}{4}.
  \end{eqnarray*}
 \end{itemize}
\end{exam}

\section{Equation II  for SI}
In this section we consider  Equation II  for $SI$.  We relabel
the real nonlinear system (\ref{rsym1}) as follows.
\begin{eqnarray}
 N(2TP_{ab}+I_b+2P_{ac})-I_c=0,\label{1ieq1}\\
2P_{ab}T^2+(2P_{ac}+I_b)T-2NP_{ab}+2P_{bc}=0.\label{1ieq2}
\end{eqnarray}

We treat the case $P_{ab}\neq 0$ and $P_{ab}= 0$ separately.

\subsection{Case $P_{ab}\neq 0$}

\begin{thm}\label{thm5.4}Equation II with $P_{ab}\neq 0$ has a solution
$$x=(Ta+b)^{-1}(aN-c),$$
where $T$ is a real solution of the following  cubic equation
\begin{equation}\label{Tab}4P_{ab}^2T^3+[4P_{ab}(2P_{ac}+I_b)]T^2+[4P_{ab}P_{bc}
+(2P_{ac}+I_b)^2]T+2P_{bc}(2P_{ac}+I_b)-2P_{ab}I_c=0. \end{equation}
and
\begin{equation}\label{Nab}N=\frac{2P_{ab}T^2+(2P_{ac}+I_b)T+2P_{bc}}{2P_{ab}}. \end{equation}
\end{thm}
\begin{proof}

If $P_{ab}\neq 0$ then by (\ref{1ieq2}) we get
\begin{equation}\label{Nab2}N=\frac{2P_{ab}T^2+(2P_{ac}+I_b)T+2P_{bc}}{2P_{ab}}. \end{equation}
Substituting the above $N$  in   (\ref{1ieq1}), we obtain
\begin{equation}\label{Tab2}4P_{ab}^2T^3+[4P_{ab}(2P_{ac}+I_b)]T^2+[4P_{ab}P_{bc}+(2P_{ac}+I_b)^2]T
+2P_{bc}(2P_{ac}+I_b)-2P_{ab}I_c=0. \end{equation}
Let $T$ be a real solution of the above cubic equation. Then the corresponding solution is
$$x=(Ta+b)^{-1}(aN-c).$$
\end{proof}

\begin{exam}\label{exam5.7}
	Consider the quadratic equation $(1+\bj)x^2 +(\bi+\bj)x-1+\bi=0$.
That is, $a=1+\bj,b=\bi+\bj$ and $c=-1+\bi$.   $P_{ab}=-1$. In this case
$T=-2, N=1$ and $$x=(Ta+b)^{-1}(aN-c)=-1.$$
\end{exam}
Combining this example with Example 3.2, we know that the set of solution of the equation $$(1+\bj)x^2 +(\bi+\bj)x-1+\bi=0$$
is
$$\{-1\}\cup \{x=x_1\bi+x_1\bj+\bk,\forall x_1\in \br\}.$$

\subsection{Case $P_{ab}=0$}
\begin{thm}\label{thm5.5}Equation II with $P_{ab}=0$ and $I_b+2P_{ac}\neq 0$
is solvable  and   $$x=(Ta+b)^{-1}(aN-c),$$  where $$N=\frac{I_c}{I_b+2P_{ac}},T= \frac{-2P_{bc}}{I_b+2P_{ac}}.$$
\end{thm}
\begin{proof}
Since $\left\langle 2x_0a+b,2x_0a+b \right\rangle=4x_0P_{ab}+I_b\neq 0$ and $P_{ab}=0$,
  we have $I_b\neq 0$.
       If $P_{ab}=0$ then by (\ref{1ieq1}) and (\ref{1ieq2}),  $(T,N)$ satisfies the real system
      \begin{eqnarray}
     N(I_b+2P_{ac})=I_c,\label{ieq1}\\
     (2P_{ac}+I_b)T=-2P_{bc}.\label{ieq2}
      \end{eqnarray}
     If $I_b+2P_{ac}\neq 0$ then $$N=\frac{I_c}{I_b+2P_{ac}},T= \frac{-2P_{bc}}{I_b+2P_{ac}}.$$
So the corresponding solution is
      $$x=(Ta+b)^{-1}(aN-c).$$
\end{proof}
\begin{exam}\label{exam5.8}
	Consider the quadratic equation $(1+\bj)x^2 +(2\bi+\bk)x+1+\bi+2\bj+\bk= 0$.
 That is, $a=1+\bj,b=2\bi+\bk$ and $c=1+\bi+2\bj+\bk$.  $P_{ab}=0, I_b+2P_{ac}=1$. In this case
$T=-2, N=-3$ and $$x=(Ta+b)^{-1}(aN-c)=-1+\frac{17}{3}\bi+\frac{1}{3}\bj+6\bk.$$
  \end{exam}

To treat  the case of $I_b+2P_{ac}=0$, we need the following  proposition.

\begin{pro}\label{prop5.3}
For the coefficients $a,b,c$ in Equation II, we assume that
$$P_{ab}=0,I_a=0,I_c=0,P_{bc}=0,I_b+2P_{ac}=0, I_b\neq 0.$$

Then we have
\begin{equation}\label{formu1}\frac{(a_2b_1+b_3)^2+(b_2-a_3b_1)^2}{\delta^2}=1,\end{equation}
\begin{equation}\label{formu2}\frac{a_3(b_2-a_3b_1)-a_2(a_2b_1+b_3)}{\delta}=-1,\end{equation}
\begin{equation}\label{formu3}\frac{a_3(a_2b_1+b_3)+a_2(b_2-a_3b_1)}{\delta}=0,\end{equation}
\begin{equation}\label{formu4}\frac{2t_2(a_2b_1+b_3)+2t_1(b_2-a_3b_1)}{\delta^2}
+\frac{b_3(b_2-a_3b_1)-b_2(a_2b_1+b_3)+2a_3t_1-2a_2t_2}{\delta}=0,\end{equation}
\begin{equation}\label{formu5}\frac{2t_2(a_3b_1-b_2)+2t_1(a_2b_1+b_3)}{\delta^2}
+\frac{(a_2b_1+b_3)b_3-b_2(a_3b_1-b_2)}{\delta}-b_1=0.\end{equation}
\end{pro}

\begin{proof}
By $a_2^2+a_3^2=1$ and $a_2b_2+a_3b_3=0$, we can easily verify Eqs.(\ref{formu1})-(\ref{formu3}).
Noting that $b_3(b_2-a_3b_1)-b_2(a_2b_1+b_3)=-b_1(a_2b_2+a_3b_3)=0$, $a_2b_1+b_3-a_2\delta=a_3(a_2b_2+a_3b_3)=0$
 and   $b_2-a_3b_1+a_3\delta=a_2(a_2b_2+a_3b_3)=0$,
  we have
  $$2t_2(a_2b_1+b_3)+2t_1(b_2-a_3b_1)+(2a_3t_1-2a_2t_2)\delta=2(a_2b_1+b_3-a_2\delta)t_2+2(b_2-a_3b_1+a_3\delta)t_1=0.$$
This proves Eq.(\ref{formu4}).
It is obvious that
$$\frac{(a_2b_1+b_3)b_3-b_2(a_3b_1-b_2)}{\delta}-b_1=\frac{b_2^2+b_3^2-b_1^2}{\delta}=\frac{-I_b}{\delta}.$$
By $a_2^2+a_3^2=1$ and $a_2b_2+a_3b_3=0$, we have
$$b_3t_1-b_2t_2+P_{ac}(a_2b_3-a_3b_2)=0.$$
Noting $a_2t_1+a_3t_2=-P_{ac}$ and $-I_b=2P_{ac}$,
 we have $$t_2(a_3b_1-b_2)+t_1(a_2b_1+b_3)+P_{ac}\delta=(a_3t_2+a_2t_1)b_1+b_3t_1-b_2t_2+P_{ac}(a_2b_3-a_3b_2+b_1)=0.$$
This proves Eq.(\ref{formu5}).
\end{proof}

\begin{thm}\label{thm5.6}Consider Equation II with $P_{ab}=0$ and $I_b+2P_{ac}=0$.
Let \begin{equation}\label{condpab02} F=t_1^2+t_2^2+(b_3t_1-b_2t_2)\delta+c_0\delta^2.\end{equation}
 Equation II
is solvable if only if $F=0$.
 If $F=0$ then  Equation II has solutions $$x=x_0+x_1\bi+x_2\bj+x_3\bk,\forall x_0,x_1\in \br,$$  where
   \begin{eqnarray}
  \label{sx2e} x_2&=&-\frac{t_2}{\delta}-\frac{(a_2b_1+b_3)}{\delta}x_0-\frac{(a_3b_1-b_2)}{\delta}x_1,\\
  \label{sx3e}x_3&=&\frac{t_1}{\delta}+\frac{(b_2-a_3b_1)}{\delta}x_0+\frac{(a_2b_1+b_3)}{\delta}x_1.
\end{eqnarray}
\end{thm}

\begin{proof}
Suppose that there is a solution $x\in SI$.
By Eq.(\ref{ieq1}) and Eq.(\ref{ieq2}), if $I_b+2P_{ac}=0$ then $$I_c=0,P_{bc}=0,I_b\neq 0.$$
In this special case, although $2x_0a+b\in \bh_s-Z(\bh_s)$,
  Eq.(\ref{ieq1}) and Eq.(\ref{ieq2})  provide no information about $N$ and $T$. So we return to the original equation.

  By Proposition \ref{prop5.1}, under the condition  $P_{ab}=0$ and $I_a=0$, we have
      \begin{eqnarray}
  (b_2-a_3b_1)x_0+(a_2b_1+b_3)x_1-\delta x_3+t_1=0,\label{6new3}\\
 (a_2b_1+b_3)x_0+(a_3b_1-b_2)x_1+\delta x_2+t_2=0.\label{6new4}
\end{eqnarray}
 Since $P_{ab}=0,a_2^2+a_3^2=1$ and $I_b\neq 0$,
  we obtain $$b_1^2-(a_3b_2-a_2b_3)^2=b_1^2-b_2^2-b_3^2+(a_2b_2+a_3b_3)^2=I_b\neq 0.$$
  This means $\delta=a_2b_3-a_3b_2+b_1\neq 0$.
  So we have
      \begin{eqnarray*}
   x_2=-\frac{(a_2b_1+b_3)}{\delta}x_0-\frac{(a_3b_1-b_2)}{\delta}x_1-\frac{t_2}{\delta},\\
  x_3=\frac{(b_2-a_3b_1)}{\delta}x_0+\frac{(a_2b_1+b_3)}{\delta}x_1+\frac{t_1}{\delta}.
\end{eqnarray*}
  Substituting  the above two formulas of $x_2$ and $x_3$ in Eq. (8),
   that is,  $$x_0^2-x_1^2+x_2^2+x_3^2+2a_2x_0x_2+2a_3x_0x_3-b_1x_1+b_2x_2+b_3x_3+c_0=0,$$
  we obtain
   \begin{eqnarray*}
  &&\Big[1+\frac{(a_2b_1+b_3)^2+(b_2-a_3b_1)^2}{\delta^2}+2\frac{a_3(b_2-a_3b_1)-a_2(a_2b_1+b_3)}{\delta}\Big]x_0^2\\
  &&+\Big[\frac{(a_2b_1+b_3)^2+(b_2-a_3b_1)^2}{\delta^2}-1\Big]x_1^2+\frac{a_3(a_2b_1+b_3)+a_2(b_2-a_3b_1)}{\delta}x_0x_1\\
    &&+\Big[\frac{2t_2(a_2b_1+b_3)+2t_1(b_2-a_3b_1)}{\delta^2}+\frac{b_3(b_2-a_3b_1)-b_3(a_2b_1+b_3)+2a_3t_1-2a_2t_2}{\delta}\Big]x_0\\
    &&+\Big[\frac{2t_2(a_3b_1-b_2)+2t_1(a_2b_1+b_3)}{\delta^2}+\frac{(a_2b_1+b_3)b_3-b_2(a_3b_1-b_2)}{\delta}-b_1\Big]x_1\\
    &&+\frac{t_1^2+t_2^2}{\delta^2}+\frac{b_3t_1-b_2t_2}{\delta}+c_0=0.\end{eqnarray*}
    By Proposition \ref{prop5.3}, if $F=0$ then  the above equation is an identity.
  Thus Equation II has solutions   $$x=x_0+x_1\bi+x_2\bj+x_3\bk,\forall x_0,x_1\in \br,$$  where
$x_2$ and $x_3$ are given by (\ref{sx2e}) and (\ref{sx3e}).
\end{proof}

\begin{exam}\label{exam5.9}
	Consider the quadratic equation $(1+\bj)x^2 +(2\bi+\bk)x-\frac{3}{4}+\frac{3}{4}\bj= 0$.
 That is, $a=1+\bj,b=2\bi+\bk$ and $c=-\frac{3}{4}+\frac{3}{4}\bj$.
  It is obvious that $P_{ab}=0,I_c=0,P_{bc}=0,I_b+2P_{ac}=0,I_b\neq 0$. Then
$\delta=3,t_1=\frac{3}{2},t_2=0, F=0$ and $$x=x_0+x_1\bi-x_0\bj+(x_1+\frac{1}{2})\bk,\forall x_0,x_1\in \br.$$
\end{exam}

\section{Verification of our examples  by companion polynomial approach}

We mention that Proposition \ref{pro1.1} still holds. We restate it as the following lemma.

\begin{lem}(c.f.\cite[Theorem 3.8]{Irene}) If $q\in Z(f)$ with $p(z)$ given by (\ref{eqcons}),   then $\Psi_{[q]}(x)$ is a divisor of $c(z)$ given by  (\ref{czpoly}) in complex number field. That is $$C(x)=Q(x)\Psi_{[q]}(x), Q(x)\in \bc[x].$$
\end{lem}
\begin{proof}
	If $q\in Z(f)$ then by \cite[Section 3.4]{sch20} there exist $h(x)\in \bh_s[x]$ such that $f(x)=h(x)(x-q)$.
	Hence $C(x)=f(x)\overline{f(x)}=h(x)(x-q)(x-\bar{q})\overline{h(x)}=h(x)\overline{h(x)}\Psi_{[q]}(x).$  Letting $Q(x)=h(x)\overline{h(x)}$ completes the proof.
\end{proof}

\begin{lem}\label{lemge}
	Let $a,d\in \bh_s$. Then the equation $ax=d$ is  solvable
	if and only if $aa^+d=d$,
	in which case all solutions are given by $$x=a^+d+(1-a^+a)y,\quad \forall y=y_0+y_1\bi+y_2\bj+y_3\bk\in \bh_s \mbox{ with } y_i\in \br.$$
\end{lem}
\begin{proof}If $a$ is invertible, then $1-a^+a=0$ and  $x=a^{-1}d$. It is obvious for the case $a=0$.  The case of $a$ being noninvertible is the same as Lemma \ref{lemmoor}.
\end{proof}
Let \begin{equation}\label{sab}S(a,d)=\{x\in \bh_s:ax=d\}.\end{equation}

\begin{thm}\label{thmcom} Suppose that  the companion polynomial of Equation II
	\begin{equation}\label{czpoly1}C(x)=2P_{ab}x^3+(2P_{ac}+I_b)x^2+2P_{bc}x+I_c\not\equiv 0\end{equation}
 Let $\Psi_{[q]}(x)=x^2-Tx+N$ be a divisor of $C(x)$.  Then  the set of solutions of Equation  II is
	$$Z(f)=\bigcup_{[q]} \{S(Ta+b,aN-c)\cap [q]\}.$$
\end{thm}
\begin{proof}
Let $q\in Z(f)$ and $T=2\Re(q),N=I_q=q\bar{q}$. Then we have
$$(Ta+b)q=aN-c.$$
Thus $q\in S(Ta+b,aN-c)$. By Lemma \ref{lemge}, we get the result.
\end{proof}

By computation, we know that the companion polynomials of Equation II in Examples 3.4, 3.5, 3.6 and 4.3 are $C(x)=0$.  We will apply Theorem \ref{thmcom} to our Examples 3.1, 3,2, 3.3 and 4.2. In these examples, we have checked that for each pair $(T,N)$ the equations $(Ta+b)x=aN-c$ are solvable.

 We present our verification procedure as follows:
\begin{itemize}
	\item[(1)]

In Example 3.1, we have  $C(x)=-4(x+\frac{1}{2})^3$, one pair $(T, N)=(-1,\frac{1}{4})$ and $$Ta+b=-1+\bi+\bj+\bk\in Z(\bh_s),\, Na-c=\frac{1}{2}-\frac{5}{2}\bi-\frac{1}{2}\bj-\frac{5}{2}\bk.$$
By  Lemma \ref{lemge}, we have
$$S(Ta+b,Na-c)=\{-\frac{3}{4}+\frac{1}{2}(y_0+y_3)+[\frac{1}{2}+\frac{1}{2}(y_1+y_2)]\bi+[-\frac{1}{2}+\frac{1}{2}(y_1+y_2)]\bj+[\frac{3}{4}+\frac{1}{2}(y_0+y_3)]\bk\}.
$$Hence
$$Z(f)=S(Ta+b,Na-c)\cap  \{x\in \bh_s: \Re(x)=\frac{-1}{2}, I_x=\frac{1}{4}\}=\{-\frac{1}{2}+\bi+\bk\}.$$

	\item[(2)]In Example 3.2(alos in  Example 4.1), we have $C(x)=-2(x-1)(x+1)^2$,  two pairs $$(T, N)=(-2,1),\,\, (T, N)=(0,-1).$$
	
For the first pair, we have $$Ta+b=-2+\bi-\bj\in \bh_s- Z(\bh_s),\, Na-c=2-\bi+\bj.$$ By  Lemma \ref{lemge}, we have $$S(Ta+b,Na-c)=\{-1\}
$$
and $$S(Ta+b,Na-c)\cap  \{x\in \bh_s: \Re(x)=-1, I_x=1\}=\{-1\}\subset Z(f).$$

For the second pair, we have $$Ta+b=\bi+\bj\in Z(\bh_s),\,Na-c=-\bi-\bj.$$
By  Lemma \ref{lemge}, we have
$$S(Ta+b,Na-c)=\{-\frac{1}{2}+\frac{1}{2}(y_0+y_3)+\frac{1}{2}(y_1+y_2)\bi+\frac{1}{2}(y_1+y_2)\bj+[\frac{1}{2}+\frac{1}{2}(y_0+y_3)]\bk\}.
$$
Hence
$$S(Ta+b,Na-c)\cap   \{x\in \bh_s: \Re(x)=0, I_x=-1\}=\{x_1\bi+x_1\bj+\bk,\forall x_1\in \br\}.$$
Therefore we have
$$Z(f)=\{-1\}\cup \{x=x_1\bi+x_1\bj+\bk,\forall x_1\in \br\}.$$

	\item[(3)]
In Example 3.3, we have  $C(x)=2(x^2-x+1)$, one pair $(T, N)=(1,1)$ and $$Ta+b=1+\bi+\bj+\bk\in Z(\bh_s),\,Na-c=\bi+\bj.$$ By  Lemma \ref{lemge}, we have $$S(Ta+b,Na-c)=\{\frac{1}{4}+\frac{1}{2}(y_0-y_2)+[\frac{1}{4}+\frac{1}{2}(y_1+y_3)]\bi+[\frac{1}{4}+\frac{1}{2}(y_2-y_0)]\bj+[-\frac{1}{4}+\frac{1}{2}(y_1+y_3)]\bk\}.
$$
$$Z(f)=S(Ta+b,Na-c)\cap  \{x\in \bh_s: \Re(x)=\frac{1}{2}, I_x=1\}=\{\frac{1}{2}+\bi+\frac{1}{2}\bk \}.$$

	\item[(4)]
In Example 4.2, we have  $C(x)=(x-1)(x+3)$, one pair $(T, N)=(-2,-3)$ and $$Ta+b=-2+2\bi-2\bj+\bk\in \bh_s-Z(\bh_s),\, Na-c=-4-\bi-5\bj-\bk.$$ By  Lemma \ref{lemge}, we have
 $$S(Ta+b,Na-c)=\{-1+\frac{17}{3}\bi+\frac{1}{3}\bj+6\bk\}
$$
Hence
$$Z(f)=\{-1+\frac{17}{3}\bi+\frac{1}{3}\bj+6\bk\}.$$
\end{itemize}

\vspace{2mm}
{\bf Acknowledgments.}\quad
 This work is supported by Natural Science Foundation of China (11871379) and  Key project of  National Natural Science Foundation  of Guangdong Province Universities (2019KZDXM025).

\end{document}